\documentclass{article}
\usepackage[initials]{amsrefs}

\usepackage{amsthm, amsmath, amssymb,graphicx}

\numberwithin{equation}{section}

\theoremstyle{plain}	     
\newtheorem{thm}{Theorem}[section] 

\newtheorem{lem}[thm]{Lemma}

\theoremstyle{definition}

\theoremstyle{remark} 
\newtheorem{rem}[thm]{Remark}
	
\newcommand{\tam}{\operatorname{tam}}
\newcommand{\tamh}{\operatorname{tamh}}

\begin{document}

\title{Applications of a duality between generalized trigonometric and hyperbolic functions II
\footnote{The work of S.T. was supported by JSPS KAKENHI Grant Number 17K05336.}}
\author{Hiroki Miyakawa and Shingo Takeuchi \\
Department of Mathematical Sciences\\
Shibaura Institute of Technology
\thanks{307 Fukasaku, Minuma-ku,
Saitama-shi, Saitama 337-8570, Japan. \endgraf
{\it E-mail address\/}: shingo@shibaura-it.ac.jp \endgraf
{\it 2010 Mathematics Subject Classification.} 
33B10 (26D05 26D07 31C45 34A34)}}
\date{}

\maketitle

\begin{abstract}
Generalized trigonometric functions and generalized hyperbolic functions can be 
converted to each other by the duality formulas previously discovered by the authors.
In this paper, we apply the duality formulas to prove dual pairs of Wilker-type 
inequalities, Huygens-type inequalities, and (relaxed) Cusa-Huygens-type inequalities 
for the generalized functions. In addition, multiple- and double-angle formulas 
not previously obtained are also given.
\end{abstract}

\textbf{Keywords:} 
Generalized trigonometric functions,
Generalized hyperbolic functions,
Mitrinovi\'{c}-Adamovi\'{c} inequalities,
Wilker inequalities,
Huygens inequalities,
Cusa-Huygens inequalities,
Multiple-angle formulas,
Double-angle formulas,
$p$-Laplacian.


\section{Introduction}

Generalized trigonometric functions (GTFs) and generalized hyperbolic
functions (GHFs) are natural mathematical generalizations of the trigonometric
and hyperbolic functions, respectively. They have been applied not only to
generalize $\pi$ and the complete elliptic integrals, but also to analyze 
nonlinear differential equations involving $p$-Laplacian (see monographs 
\cites{Dosly2005,Lang2011} and survey \cite{YHWL2019}, and the references
given there).

Although GTFs and GHFs have been actively studied,
they have been treated separately (e.g.,
\cites{Dosly2005,Klen,MSZH,Neuman2015,YHQ,YHWL2019}). In our 
previous work [9], the authors proved duality formulas that can transform GTFs and GHFs into 
each other. As an application, we were able to construct generalized inequalities of 
the classical Mitrinovi\'{c}-Adamovi\'{c} inequalities to GTFs and GHFs such that 
they are 
dual pairs to each other in the sense explained in the next section.

In this paper, following \cite{Miyakawa-Takeuchi}, we will generalize the old and 
vigorously studied Wilker inequalities, Huygens inequalities and (relaxed) Cusa-
Huygens inequalities to GTFs and GHFs. In fact, previous works, e.g., 
\cites{Klen,MSZH,Neuman2015,YHQ,YHWL2019}, have made various generalizations of 
these inequalities, but the trigonometric and hyperbolic versions are not in dual pairs. 
On the other hand, the pairs we create in the present paper are dual to each other.

This paper is organized as follows. Section \ref{sec:preparation} summarizes the 
definitions of GTFs and GHFs and their properties, including the duality formulas
obtained in \cite{Miyakawa-Takeuchi}. 
Here, the conditions imposed on the parameters contained in these functions are 
more extended than usual. This extension reveals the duality between both 
generalized functions. In Section \ref{sec:inequalities}, we generalize the classical 
Wilker inequalities, Huygens inequalities, and (relaxed) Cusa-Huygens inequalities to 
GTFs and GHFs. It should be noted that the pairs of inequalities obtained there are 
dual to each other. In Section \ref{sec:formulas}, as a further application of the duality 
formulas, we provide multiple- and double-angle formulas for GTFs and GHFs. 
Although some formulas have already been obtained in previous studies
(cf. \cite{Takeuchi2016} and Table \ref{hyo} in Section \ref{sec:formulas}), we give 
formulas for parameters for which no formulas were previously known.

\section{Preparation}
\label{sec:preparation}

In this section, we summarize the definitions and
some properties of GTFs and GHFs (see \cite{Miyakawa-Takeuchi} for more details).
The relationship between GTFs and GHFs can be seen by making the range of 
parameters in the functions wider than the conventional definition.

Let us assume
\begin{equation}
\label{eq:pq}
\frac{q}{q+1}<p<\infty,\quad 0<q<\infty,
\end{equation}
and
$$F_{p,q}(y):=\int_0^y \frac{dt}{(1-t^q)^{1/p}}, \quad y \in [0,1).$$
We will denote by $\sin_{p,q}$ the inverse function of $F_{p,q}$, i.e.,
$$\sin_{p,q}{x}:=F_{p,q}^{-1}(x).$$
Clearly, $\sin_{p,q}{x}$ is monotone increasing on
$[0,\pi_{p,q}/2)$ onto $[0,1)$,
where
\begin{align*}
\pi_{p,q}:&=2F_{p,q}(1)=2\int_0^1 \frac{dt}{(1-t^q)^{1/p}}\\
&=
\begin{cases}
(2/q)B(1-1/p,1/q), & 1<p<\infty,\\
\infty, & q/(q+1)<p \leq 1,
\end{cases}
\end{align*}
and $B$ is the beta function.
In almost all literature dealing with GTFs, 
the parameters $p,\ q$ are assumed to be $p,\ q>1$,
but we here allow them to be $p,\ q \leq 1$.
Note that the condition $q/(q+1)<p \leq 1$ implies that $\sin_{p,q}{x}$ is monotone
increasing on the \textit{infinite} interval 
$[0,\infty)$ and no longer similar to $\sin{x}$, but to $\tanh{x}$
(Figure \ref{fig:sin}).

\begin{figure}[htbp]
\begin{center}
\includegraphics[width=0.4\linewidth]{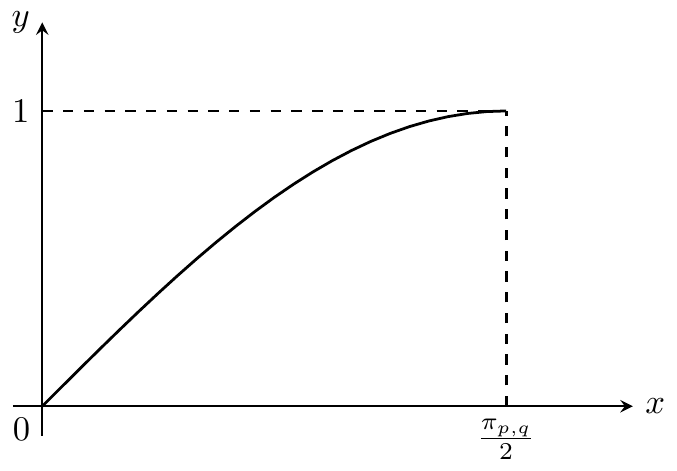} \qquad
\raisebox{1.5mm}{\includegraphics[width=0.4\linewidth]{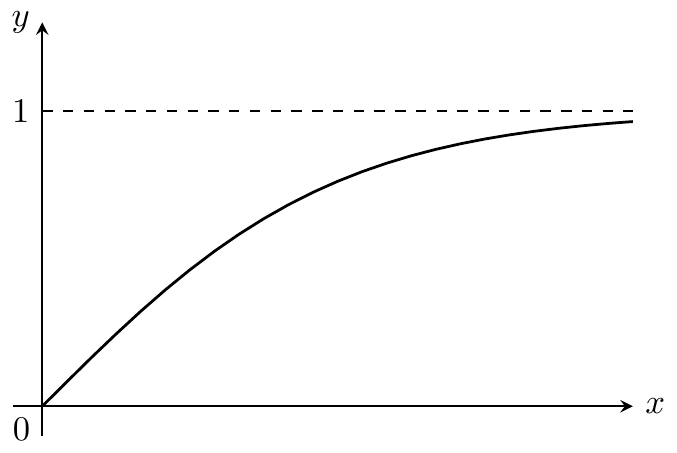}}
\caption{The graphs of $\sin_{p,q}{x}$: $p>1$ (left) and $q/(q+1)<p \leq 1$ (right).}
\label{fig:sin}
\end{center}
\end{figure}

Since $\sin_{p,q}{x} \in C^1(0,\pi_{p,q}/2)$,
we also define $\cos_{p,q}{x}$ by 
$$\cos_{p,q}{x}:=\frac{d}{dx}(\sin_{p,q}{x}).$$
Then, it follows that 
\begin{equation}
\label{eq:Tpythagoras}
\cos_{p,q}^p{x}+\sin_{p,q}^q{x}=1.
\end{equation}
In case $(p,q)=(2,2)$, it is obvious that $\sin_{p,q}{x},\ \cos_{p,q}{x}$ 
and $\pi_{p,q}$ are reduced to the ordinary $\sin{x},\ \cos{x}$ and $\pi$,
respectively. Therefore these functions and the constant are called
\textit{generalized trigonometric functions} (GTFs)
and the \textit{generalized $\pi$}, respectively. 
It is easy to check that $u=\sin_{p,q}{x}$ is a solution of
the initial value problem of $p$-Laplacian
\begin{equation}
\label{eq:ivp}
(|u'|^{p-2}u')'+\frac{(p-1)q}{p} |u|^{q-2}u=0, \quad u(0)=0,\ u'(0)=1,
\end{equation}
which is closely related to 
the eigenvalue problem of $p$-Laplacian.

In a similar way, we assume \eqref{eq:pq} and
$$G_{p,q}(y):=\int_0^y \frac{dt}{(1+t^q)^{1/p}}, \quad y \in [0,\infty).$$
We will denote by $\sinh_{p,q}$ the inverse function of $G_{p,q}$, i.e.,
$$\sinh_{p,q}{x}:=G_{p,q}^{-1}(x).$$
Clearly, $\sinh_{p,q}{x}$ is monotone increasing
on $[0,\pi_{r,q}/2)$ onto $[0,\infty)$,
where $r$ is the positive constant determined by
\begin{equation}
\label{eq:rdefi}
\frac{1}{p}+\frac{1}{r}=1+\frac{1}{q}, \quad \mbox{i.e.}, \quad
r=\frac{pq}{pq+p-q}.
\end{equation}
Indeed, by $1+t^q=1/(1-s^q)$,
\begin{align*}
\lim_{y \to \infty}G_{p,q}(y)=
\int_0^\infty \frac{dt}{(1+t^q)^{1/p}}
=\int_0^1 \frac{ds}{(1-s^q)^{1/r}}
=\frac{\pi_{r,q}}{2}.
\end{align*}
The important point to note here is that for a fixed $q \in (0,\infty)$,
if $r=r_q(p)$ is regarded as a function of $p$, then 
\begin{gather}
\mbox{$r_q$ is bijective from $(q/(q+1),\infty)$ to itself, and}
\label{eq:r(p)}\\
r_q(r_q(p))=p. \label{eq:r(r)=p}
\end{gather}
In particular, $\pi_{r,q}$ has been defined under \eqref{eq:pq}.
If $r>1$, i.e., $p<q$, then $\sinh_{p,q}{x}$ is defined in
the \textit{bounded} interval $[0,\pi_{r,q}/2)$ with 
$\lim_{x \to \pi_{r,q}/2}\sinh_{p,q}{x}=\infty$ and
no longer similar to $\sinh{x}$, but to $\tan{x}$
(Figure \ref{fig:sinh}).

\begin{figure}[htbp]
\begin{center}
\includegraphics[width=0.3\linewidth]{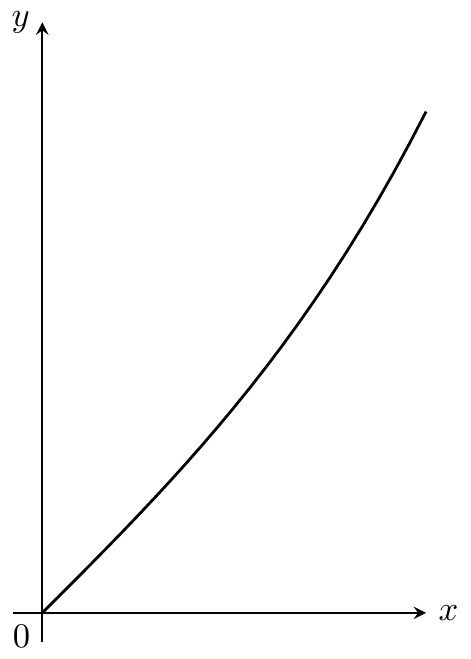} \qquad
\raisebox{-1.5mm}{\includegraphics[width=0.3\linewidth]{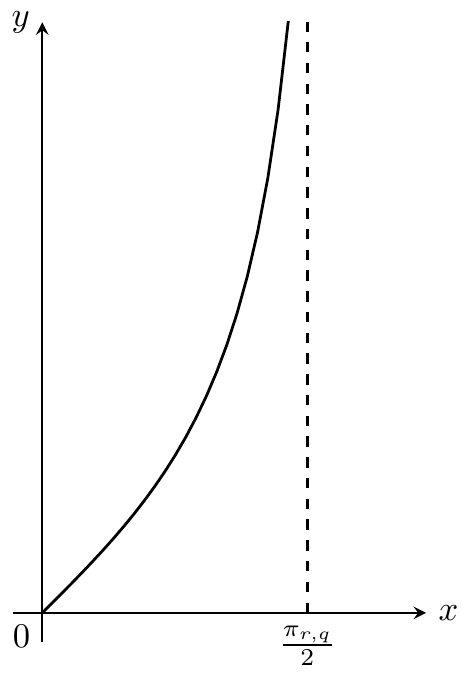}}
\caption{The graphs of $\sinh_{p,q}{x}$: $r \leq 1$ (left) and $r>1$ (right).}
\label{fig:sinh}
\end{center}
\end{figure}

Since $\sinh_{p,q}{x} \in C^1(0,\pi_{r,q}/2)$,
we also define $\cosh_{p,q}{x}$ by 
$$\cosh_{p,q}{x}:=\frac{d}{dx}(\sinh_{p,q}{x}).$$
Then, it follows that 
\begin{equation}
\label{eq:Hpythagoras}
\cosh_{p,q}^p{x}-\sinh_{p,q}^q{x}=1.
\end{equation}
In case $(p,q)=(2,2)$, it is obvious that $\sinh_{p,q}{x},\ \cosh_{p,q}{x}$ 
and the interval $[0,\pi_{r,q}/2)$ are reduced to 
$\sinh{x},\ \cosh{x}$ and $[0,\infty)$, respectively. 
Therefore these functions are called
\textit{generalized hyperbolic functions} (GHFs). 
Just as $\sin_{p,q}{x}$ satisfies \eqref{eq:ivp}, 
$u=\sinh_{p,q}{x}$ is a solution of 
the initial value problem of $p$-Laplacian
$$(|u'|^{p-2}u')'-\frac{(p-1)q}{p} |u|^{q-2}u=0, \quad u(0)=0,\ u'(0)=1.$$

We generalize the tangent and hyperbolic tangent functions in two ways.
These functions are often generalized as
$$\tan_{p,q}{x}:=\frac{\sin_{p,q}{x}}{\cos_{p,q}{x}}, \quad
\tanh_{p,q}{x}:=\frac{\sinh_{p,q}{x}}{\cosh_{p,q}{x}}$$
(e.g., 
\cites{Dosly2005,Edmunds2012,Klen,Lang2011,MSZH,Neuman2015,YHQ,YHWL2019}).
However, for practical purposes, 
the following modified functions are more convenient than 
the functions above:
$$\tam_{p,q}x:=\frac{\sin_{p,q}x}{\cos_{p,q}^{p/q}x}, \quad
\tamh_{p,q}x:=\frac{\sinh_{p,q}x}{\cosh_{p,q}^{p/q}x}.$$
These modified functions were first introduced in \cites{Miyakawa-Takeuchi,Takeuchi2016}
with the symbols $\tau_{p,q},\ \tilde{\tau}_{p,q}$, respectively.
Note that if $p=q$, then $\tam_{p,q}{x}=\tan_{p,q}{x}$ and  
$\tamh_{p,q}{x}=\tanh_{p,q}{x}$.

In \cite{Miyakawa-Takeuchi}, 
we proved the following duality properties between GTFs and GHFs.
This property remains important in the present paper.

\begin{thm}[\cite{Miyakawa-Takeuchi}]\label{lem:GHFGTFrelation}
Let $p$ and $q$ satisfy \eqref{eq:pq} and $r$ be the positive number defined as 
\eqref{eq:rdefi}. Then, for $x\in[0,\pi_{p,q}/2)$,
\begin{align*}
&\sin_{p,q}x=\frac{\sinh_{r,q}x}{\cosh_{r,q}^{r/q}x}=\tamh_{r,q}x,\\
&\cos_{p,q}x=\frac{1}{\cosh_{r,q}^{r/p}x},\\
&\tam_{p,q}x=\sinh_{r,q}x.
\end{align*}
\end{thm}

\begin{thm}[\cite{Miyakawa-Takeuchi}]\label{lem:GTFGHFrelation}
Let $p$ and $q$ satisfy \eqref{eq:pq} and $r$ be the positive number defined as 
\eqref{eq:rdefi}. Then, for $x\in[0,\pi_{r,q}/2)$,
\begin{align*}
&\sinh_{p,q}x=\frac{\sin_{r,q}x}{\cos_{r,q}^{r/q}x}=\tam_{r,q}x,\\
&\cosh_{p,q}x=\frac{1}{\cos_{r,q}^{r/p}x},\\
&\tamh_{p,q}x=\sin_{r,q}x.
\end{align*}
\end{thm}

\begin{rem}
In \cite{Miyakawa-Takeuchi}, we have supposed $q$ to be $1<q<\infty$.
However, the proofs therein are perfectly valid in the case $0<q<\infty$ as well.
The same is true for Theorem \ref{thm:MA} below.
\end{rem}

Theorems \ref{lem:GHFGTFrelation} and \ref{lem:GTFGHFrelation} tell us the 
counterparts to GHFs of the properties already known for GTFs, and vice versa.
For example, Theorem \ref{lem:GHFGTFrelation} immediately converts 
\eqref{eq:Tpythagoras} into \eqref{eq:Hpythagoras} (with $p$ replaced by $r$);
that is,
\begin{equation}
\label{eq:p+q=1'}
\cos_{p,q}^p{x}+\sin_{p,q}^q{x}=1
\end{equation}
into
\begin{equation}
\label{eq:p-q=1'}
\cosh_{r,q}^r{x}-\sinh_{r,q}^q{x}=1;
\end{equation}
and Theorem \ref{lem:GTFGHFrelation} (with \eqref{eq:r(r)=p}) does vice versa.
Hence, it follows from \eqref{eq:r(p)} that 
\eqref{eq:p+q=1'} and \eqref{eq:p-q=1'} correspond one-to-one.
In this sense, we say that inequalities \eqref{eq:p+q=1'} and \eqref{eq:p-q=1'} 
(i.e., \eqref{eq:Hpythagoras} with $p$ replaced by $r$)
are
\textit{dual} to each other.

Moreover, using our theorems, the authors \cite{Miyakawa-Takeuchi}*{Theorem 1.3} 
have shown the generalization 
of Mitrinovi\'{c}-Adamovi\'{c} inequalities.
The generalized inequalities will be applied in the next section, and are 
discussed now.
The classical Mitrinovi\'{c}-Adamovi\'{c} inequalities are as follows:
\begin{gather*}
\cos^{1/3}{x}<\frac{\sin{x}}{x}<1, \quad x \in \left(0,\frac{\pi}{2}\right),\\
\cosh^{1/3}{x}<\frac{\sinh{x}}{x}<\cosh{x}, \quad x \in \left(0,\infty\right).
\end{gather*}
The latter is also called the Lazarevi\'c inequality.
Kl\'{e}n et al. \cite{Klen}*{Theorems 3.6 and 3.8}
extend these inequalities to the one-parameter case: for $p \in (1,\infty)$,
\begin{gather*}
\cos_p^{1/(p+1)}{x}<\frac{\sin_p{x}}{x}<1, \quad x \in \left(0,\frac{\pi_p}{2}\right),\\
\cosh_p^{1/(p+1)}{x}<\frac{\sinh_p{x}}{x}<\cosh_p{x}, \quad x \in \left(0,\infty\right),
\end{gather*}
where $\sin_p{x}:=\sin_{p,p}{x}$ and the other functions are defined in the same way.
Moreover, Ma et al. \cite{MSZH}*{Lemma 2} obtain 
the inequalities for the two-parameter case: for $p,\ q \in (1,\infty)$,
\begin{gather}
\cos_{p,q}^{1/(q+1)}{x}<\frac{\sin_{p,q}{x}}{x}<1, \quad x \in \left(0,\frac{\pi_{p,q}}{2}\right), \label{eq:MAtrigo'}\\
\cosh_{p,q}^{1/(q+1)}{x}<\frac{\sinh_{p,q}{x}}{x}<\cosh_{p,q}{x} 
\quad \mbox{for appropriate $x$}.
\label{eq:MAhyper'}
\end{gather}
The proofs of Kl\'en et al. \cite{Klen} and Ma et al. \cite{MSZH} 
are similar, and both prove the inequalities for the trigonometric case 
and the hyperbolic case separately in the same way.

However, \eqref{eq:MAtrigo'} and \eqref{eq:MAhyper'} (with $p$ replaced by $r$)
are not dual to each other. 
A dual pair of Mitrinovi\'{c}-Adamovi\'{c}-type inequalities is as follows:

\begin{thm}[Mitrinovi\'{c}-Adamovi\'{c}-type inequalities with duality, \cite{Miyakawa-Takeuchi}]\label{thm:MA}
Let $p$ and $q$ satisfy \eqref{eq:pq} and $r$ be the positive number defined as \eqref{eq:rdefi}. 
Then,
\begin{gather}
\cos_{p,q}^{1/(q+1)}x<\frac{\sin_{p,q}x}{x}<1, \quad x\in \left(0,\frac{\pi_{p,q}}{2}\right), \label{eq:MAtrigo}\\
\cosh_{p,q}^{1/(q+1)}x<\frac{\sinh_{p,q}x}{x}<\cosh_{p,q}^{p/q}x,
\quad x\in \left(0,\frac{\pi_{r,q}}{2}\right).\label{eq:MAhyper}
\end{gather}
Moreover, \eqref{eq:MAtrigo} and \eqref{eq:MAhyper} (with $p$ replaced by $r$)
are dual to each other.
\end{thm}

\begin{rem}
If $p=q$, then \eqref{eq:MAtrigo} and \eqref{eq:MAhyper} are
equal to \eqref{eq:MAtrigo'} and \eqref{eq:MAhyper'}; hence, to
the one-parameter ones above.
\end{rem}

In our approach in \cite{Miyakawa-Takeuchi}, 
Theorem \ref{thm:MA} allows us to obtain the inequalities 
over the wider range \eqref{eq:pq} of parameters,
and \eqref{eq:MAhyper} immediately follows from \eqref{eq:MAtrigo} 
by Theorem \ref{lem:GHFGTFrelation}.

\section{Dual pairs of inequalities}
\label{sec:inequalities}

In this section,
we generalize the Wilker, Huygens, and (relaxed) Cusa-Huygens inequalities 
for GTFs and GHFs to a form with duality 
using our duality formulas (Theorems \ref{lem:GHFGTFrelation} and \ref{lem:GTFGHFrelation}),
just as we generalized the Mitrinovi\'{c}-Adamovi\'{c} inequalities
as Theorem \ref{thm:MA}.

\subsection{Wilker-type inequalities}

The classical Wilker inequalities are as follows:
\begin{gather*}
\left(\frac{\sin{x}}{x}\right)^2+\frac{\tan{x}}{x}>2, \quad x \in \left(0,\frac{\pi}{2}\right),\\
\left(\frac{\sinh{x}}{x}\right)^2+\frac{\tanh{x}}{x}>2, \quad x \in \left(0,\infty\right).
\end{gather*}
Kl\'{e}n et al. \cite{Klen}*{Corollary 3.19} and Yin et al. \cite{YHQ}*{Theorem 3.1}
extend these inequalities to the one-parameter case: for $p \in (1,\infty)$,
\begin{gather*}
\left(\frac{\sin_p{x}}{x}\right)^p+\frac{\tan_p{x}}{x}>2, \quad x \in \left(0,\frac{\pi_p}{2}\right),\\
\left(\frac{\sinh_p{x}}{x}\right)^p+\frac{\tanh_p{x}}{x}>2, \quad x \in \left(0,\infty\right).
\end{gather*}
Moreover, Neuman \cite{Neuman2015}*{Corollary 6.3 (6.13)} obtains 
the inequalities for the two-parameter case: for $p,\ q \in (1,\infty)$,
\begin{gather}
\left(\frac{\sin_{p,q}{x}}{x}\right)^q+\frac{\tan_{p,q}{x}}{x}>2, \quad x \in \left(0,\frac{\pi_{p,q}}{2}\right), \label{eq:wilker}\\
\left(\frac{\sinh_{p,q}{x}}{x}\right)^q+\frac{\tanh_{p,q}{x}}{x}>2 \quad
\mbox{for appropriate $x$}.
\label{eq:wilkerh}
\end{gather}

However, \eqref{eq:wilker} and \eqref{eq:wilkerh} 
(with $p$ replaced by $r$) are not dual to each other. 
A dual pair of Wilker-type inequalities is as follows:

\begin{thm}[Wilker-type inequalities with duality]\label{thm:Wil2pMiya}
Let $p$ and $q$ satisfy \eqref{eq:pq} and $r$ be the positive number defined as \eqref{eq:rdefi}. 
Then,
\begin{gather}
\left(\frac{\sin_{p,q}x}{x}\right)^p+\left(\frac{\tam_{p,q}x}{x}\right)^r>2, \quad
x \in \left(0,\frac{\pi_{p,q}}{2}\right), \label{eq:wilkertrigo}\\
\left(\frac{\sinh_{p,q}x}{x}\right)^p+\left(\frac{\tamh_{p,q}x}{x}\right)^r>2, \quad
x \in \left(0,\frac{\pi_{r,q}}{2}\right). \label{eq:wilkerhyper}
\end{gather}
Moreover, \eqref{eq:wilkertrigo} and \eqref{eq:wilkerhyper} (with $p$ replaced by $r$)
are dual to each other.
\end{thm}

\begin{rem}
If $p=q$, then \eqref{eq:wilkertrigo} and \eqref{eq:wilkerhyper} are
equal to \eqref{eq:wilker} and \eqref{eq:wilkerh}; hence, 
to the one-parameter ones above, respectively.
\end{rem}

\begin{proof}
We prove \eqref{eq:wilkertrigo}. From the inequality of arithmetic and geometric 
means and \eqref{eq:MAtrigo} in Theorem \ref{thm:MA}, it follows that
\begin{align*}
\left(\frac{\sin_{p,q}x}{x}\right)^p+\left(\frac{\tam_{p,q}x}{x}\right)^r
& \geq 2\left(\frac{\sin_{p,q}x}{x}\right)^{p/2}\left(\frac{\tam_{p,q}x}{x}\right)^{r/2}\\
& =2\left(\frac{\sin_{p,q}x}{x}\right)^{p/2+r/2}\left(\frac{1}{\cos_{p,q}^{p/q}x}\right)^{r/2}\\
& >2\left(\frac{\sin_{p,q}x}{x}\right)^{p/2+r/2}\left(\frac{\sin_{p,q}x}{x}\right)^{-rp(q+1)/2q}\\
& =2.
\end{align*}

Next we show \eqref{eq:wilkerhyper}. 
For any $x \in (0,\pi_{p,q}/2)$, we have proved \eqref{eq:wilkertrigo}.
Then, Theorem \ref{lem:GHFGTFrelation} gives the dual inequality to \eqref{eq:wilkertrigo}:
$$\left(\frac{\tamh_{r,q}x}{x}\right)^p+\left(\frac{\sinh_{r,q}x}{x}\right)^r>2.$$
Owing to \eqref{eq:r(p)} and \eqref{eq:r(r)=p}, 
this means \eqref{eq:wilkerhyper}.
\end{proof}

\subsection{Huygens-type inequalities}

The classical Huygens inequalities are as follows:
\begin{gather*}
\frac{2\sin{x}}{x}+\frac{\tan{x}}{x}>3, \quad x \in \left(0,\frac{\pi}{2}\right),\\
\frac{2\sinh{x}}{x}+\frac{\tanh{x}}{x}>3, \quad x \in \left(0,\infty\right).
\end{gather*}
Kl\'{e}n et al. \cite{Klen}*{Theorem 3.16}
extend these inequalities to the one-parameter case: for $p \in (1,\infty)$,
\begin{gather*}
\frac{p\sin_p{x}}{x}+\frac{\tan_p{x}}{x}>p+1, \quad x \in \left(0,\frac{\pi_p}{2}\right),\\
\frac{p\sinh_p{x}}{x}+\frac{\tanh_p{x}}{x}>p+1, \quad x \in \left(0,\infty\right).
\end{gather*}
Moreover, Neuman \cite{Neuman2015}*{Corollary 6.3 (6.14)} obtains 
the inequalities for the two-parameter case: for $p,\ q \in (1,\infty)$,
\begin{gather}
\frac{q\sin_{p,q}{x}}{x}+\frac{\tan_{p,q}{x}}{x}>q+1, \quad x \in \left(0,\frac{\pi_{p,q}}{2}\right), \label{eq:cusa}\\
\frac{q\sinh_{p,q}{x}}{x}+\frac{\tanh_{p,q}{x}}{x}>q+1 \quad
\mbox{for appropriate $x$}.
\label{eq:cusah}
\end{gather}

However, \eqref{eq:cusa} and \eqref{eq:cusah} (with $p$ replaced by $r$)
are not dual to each other. 
A dual pair of Huygens-type inequalities is as follows:

\begin{thm}[Huygens-type inequalities with duality]\label{thm:HuyMiya}
Let $p$ and $q$ satisfy \eqref{eq:pq} and $r$ be the positive number defined as \eqref{eq:rdefi}. 
Then,
\begin{gather}
\frac{p\sin_{p,q}x}{x}+\frac{r\tam_{p,q}x}{x}>p+r, \quad 
x \in \left(0,\frac{\pi_{p,q}}{2}\right), \label{eq:Huygenstrigo}\\
\frac{p\sinh_{p,q}x}{x}+\frac{r\tamh_{p,q}x}{x}>p+r, \quad
x \in \left(0,\frac{\pi_{r,q}}{2}\right).\label{eq:Huygenshyper}
\end{gather}
Moreover, \eqref{eq:Huygenstrigo} and \eqref{eq:Huygenshyper} 
(with $p$ replaced by $r$) are dual to each other.
\end{thm}

\begin{rem}
If $p=q$, then \eqref{eq:Huygenstrigo} and \eqref{eq:Huygenshyper} are
equal to \eqref{eq:cusa} and \eqref{eq:cusah}; hence, 
to the one-parameter ones above, respectively.
\end{rem}

\begin{proof}
Let $\alpha,\ \beta$ be
$$\alpha=\frac{p(q+1)}{pq+p-q}=1+\frac{r}{p},\quad \beta=\frac{p(q+1)}{q}=1+\frac{p}{r}.$$
Since $\alpha,\ \beta >1$ and $1/\alpha+1/\beta=1$,
the following Young inequality holds true for positive numbers $A,\ B$:
$$A+B \geq (\alpha A)^{1/\alpha}(\beta B)^{1/\beta}.$$
Therefore, by this inequality and \eqref{eq:MAtrigo} in Theorem \ref{thm:MA},
\begin{align*}
\frac{p\sin_{p,q}x}{x}+\frac{r\tam_{p,q}x}{x}
& \geq \left(\alpha \frac{p\sin_{p,q}x}{x}\right)^{1/\alpha}
\left(\beta \frac{r\tam_{p,q}x}{x}\right)^{1/\beta}\\
& = (\alpha p)^{1/\alpha} (\beta r)^{1/\beta}
\frac{\sin_{p,q}x}{x} \left(\frac{1}{\cos_{p,q}^{p/q}x}\right)^{1/\beta}\\
& >  (p+r)^{1/\alpha} (r+p)^{1/\beta}
\frac{\sin_{p,q}x}{x} \left(\frac{\sin_{p,q}x}{x}\right)^{-\frac{p(q+1)}{\beta q}}\\
&=p+r.
\end{align*}

Next we show \eqref{eq:Huygenshyper}. 
For any $x \in (0,\pi_{p,q}/2)$, we have proved \eqref{eq:Huygenstrigo}.
Then, Theorem \ref{lem:GHFGTFrelation} gives the dual inequality to \eqref{eq:Huygenstrigo}:
$$\frac{p\tamh_{r,q}x}{x}+\frac{r\sinh_{r,q}x}{x}>p+r.$$
Owing to \eqref{eq:r(p)} and \eqref{eq:r(r)=p}, 
this means \eqref{eq:Huygenshyper}.
\end{proof}

\subsection{Relaxed Cusa-Huygens-type inequalities}

The classical Cusa-Huygens inequalities are as follows:
\begin{gather*}
\frac{\sin{x}}{x}<\frac{2+\cos{x}}{3}, \quad x \in \left(0,\frac{\pi}{2}\right),\\\frac{\sinh{x}}{x}<\frac{2+\cosh{x}}{3}, \quad x \in \left(0,\infty\right).
\end{gather*}
Ma et al. \cite{MSZH}*{Theorems 2 and 3} obtain 
the inequalities for the two-parameter case: for $p,\ q \in (1,2]$,
\begin{gather}
\frac{\sin_{p,q}{x}}{x}<\frac{q+\cos_{p,q}{x}}{q+1}, \quad x \in \left(0,\frac{\pi_{p,q}}{2}\right), \label{eq:cusahuygens}\\
\frac{\sinh_{p,q}{x}}{x}<\frac{q+\cosh_{p,q}{x}}{q+1} \quad
\mbox{for appropriate $x$}.
\label{eq:cusahuygensh}
\end{gather}
The inequalities for the one parameter $p=q \in (1,2]$ are given by 
Kl\'{e}n et al. \cite{Klen}*{Theorems 3.22 and 3.24}.
Unfortunately, these generalized inequalities are shown only for $p,\ q \in (1,2]$,
and \eqref{eq:cusahuygens} and \eqref{eq:cusahuygensh} (with $p$ replaced by $r$) are 
not dual to each other. 

We hope to find inequalities that hold for $p,\ q$ satisfying \eqref{eq:pq} 
and are dual to each other.
Therefore, consider the following relaxed inequalities instead of the classical Cusa-Huygens inequalities:
\begin{gather*}
\frac{\sin{x}}{x}<\sqrt{\frac{2+\cos^2{x}}{3}}, \quad x \in \left(0,\frac{\pi}{2}\right),\\\frac{\sinh{x}}{x}<\sqrt{\frac{2+\cosh^2{x}}{3}}, \quad x \in \left(0,\infty\right).
\end{gather*}
Neuman \cite{Neuman2015}*{Theorem 6 (6.7), (6.9)} generalizes 
the inequalities to the two-parameter case: for $p,\ q \in (1,\infty)$,
\begin{gather}
\frac{\sin_{p,q}{x}}{x}<\left(\frac{q+\cos_{p,q}^p{x}}{q+1}\right)^{1/p}, \quad x \in \left(0,\frac{\pi_{p,q}}{2}\right), \label{eq:cusa'}\\
\frac{\sinh_{p,q}{x}}{x}<\left(\frac{q+\cosh_{p,q}^p{x}}{q+1}\right)^{1/p} \quad
\mbox{for appropriate $x$}.
\label{eq:cusah'}
\end{gather}

However, \eqref{eq:cusa'} and \eqref{eq:cusah'} (with $p$ replaced by $r$) are 
not dual to each other. A dual pair of Cusa-Huygens-type inequalities is as follows:

\begin{thm}[Relaxed Cusa-Huygens-type inequalities with duality]\label{thm:CusaMiya}
Let $p$ and $q$ satisfy \eqref{eq:pq} and $r$ be the positive number defined as \eqref{eq:rdefi}. 
Then,
\begin{gather}
\frac{\sin_{p,q}x}{x}<\left(\frac{p+r\cos_{p,q}^px}{p+r}\right)^{1/q}, \quad
x \in \left(0,\frac{\pi_{p,q}}{2}\right), \label{eq:CusaGTF}\\
\frac{\sinh_{p,q}x}{x}<\left(\frac{p+r\cosh_{p,q}^px}{p+r}\right)^{1/q}, \quad
x \in \left(0,\frac{\pi_{r,q}}{2}\right). \label{eq:CusaGHF}
\end{gather}
Moreover, \eqref{eq:CusaGTF} and \eqref{eq:CusaGHF} (with $p$ replaced by $r$) are dual to each other.
\end{thm}

\begin{rem}
If $p=q$, then \eqref{eq:CusaGTF} and \eqref{eq:CusaGHF} are
equal to \eqref{eq:cusa'} and \eqref{eq:cusah'}, respectively.
\end{rem}

\begin{proof}
We prove 
$$\frac{(p+r+rx^q)\sin_{p,q}^qx}{x^q}<p+r,$$
which is equivalent to \eqref{eq:CusaGTF}. 
Let $f(x):=p+r-(p+r+rx^q)\sin_{p,q}^qx/x^q$. Then,
\begin{align*}
f'(x)=\frac{q\sin_{p,q}^qx}{x^{q+1}}\left(p+r-\frac{(p+r)x\cos_{p,q}x}{\sin_{p,q}x}-\frac{rx^{q+1}\cos_{p,q}x}{\sin_{p,q}x}\right).
\end{align*}
From \eqref{eq:MAtrigo} in Theorem \ref{thm:MA}, it follows that
$$\frac{x\cos_{p,q}x}{\sin_{p,q}x}<\cos_{p,q}^{q/(q+1)}x,\ \ \ \frac{x^{q+1}\cos_{p,q}x}{\sin_{p,q}x}<1-\cos_{p,q}^px.$$
Therefore,
\begin{align*}
f'(x)>\frac{q\sin_{p,q}^qx}{x^{q+1}}\left(p-(p+r)\cos_{p,q}^{q/(q+1)}x+r\cos_{p,q}^{p}x\right).
\end{align*}
Now let $g(t)=p-(p+r)t^{q/(q+1)}+rt^p$. Then,
\begin{align*}
g'(t)=prt^{-1/(q+1)}(t^{p-q/(q+1)}-1)
\end{align*}
Since $q/(q+1)<p$, we see $g'(t)<0$. Therefore, $g(t)>\lim_{t \to 1-0}g(t)=0$ and 
$$f'(x)>\frac{q\sin_{p,q}^qx}{x^{q+1}}g(\cos_{p,q}x)>0.$$ 
Moreover, $f(x)>\lim_{x\to+0}f(x)=0$, which means \eqref{eq:CusaGTF}.

Next we show \eqref{eq:CusaGHF}. 
For any $x \in (0,\pi_{p,q}/2)$, we have proved \eqref{eq:CusaGTF}.
Then, Theorem \ref{lem:GHFGTFrelation} gives the dual inequality to \eqref{eq:CusaGTF}:
$$\frac{\sinh_{r,q}x}{x\cosh_{r,q}^{r/q}x}<\left(\frac{p+r\cosh_{r,q}^{-r}x}{p+r}\right)^{1/q};$$
hence,
$$\frac{\sinh_{r,q}x}{x}<\left(\frac{p\cosh_{r,q}^{r}x+r}{p+r}\right)^{1/q}.$$
Owing to \eqref{eq:r(p)} and \eqref{eq:r(r)=p}, 
this means \eqref{eq:CusaGHF}.
\end{proof}
\section{Multiple- and double-angle formulas}
\label{sec:formulas}

Several multiple- and double-angle formulas for GTFs and GHFs are already known
(see \cite{Miyakawa-Takeuchi}*{Theorems 1.4 and 1.6} 
and \cite{Takeuchi2016}*{Theorem 1.1} for multiple-angle formulas;
Table \ref{hyo} for double-angle formulas). 
In this section,
we apply the duality formulas (Theorems \ref{lem:GHFGTFrelation} and \ref{lem:GTFGHFrelation}) to obtain 
multiple- and double-angle formulas
which are not covered in \cite{Miyakawa-Takeuchi} and \cite{Takeuchi2016},
for these generalized functions. 

\begin{table}[htb]
\begin{center}
\caption{The parameters for which the double angle formulas of GTF have been obtained.}\label{hyo}
{\footnotesize
  \begin{tabular}{clll} \hline
    $q$ & $(q/(q-1),2)$ & $(2,q)$ & $(q/(q-1),q)$\\ \hline 
    $2$ & $(2,2)$ Abu al-Wafa & $(2,2)$ Abu al-Wafa & $(2,2)$ Abu al-Wafa\\
    $3$ & $(3/2,2)$ Miyakawa-Takeuchi \cite{Miyakawa-Takeuchi} & $(2,3)$ Cox-Shurman \cite{Cox2005} & $(3/2,3)$ Dixon \cite{Dixon1890} \\
    $4$ & $(4/3,2)$ Sato-Takeuchi \cite{Sato-Takeuchi2020} & $(2,4)$ Fagnano & $(4/3,4)$ Edmunds\ et\ al. \cite{Edmunds2012} \\ 
    $6$ & $(6/5,2)$ Takeuchi \cite{Takeuchipre} & $(2,6)$ Shinohara \cite{Shinohara2017} & $(6/5,6)$ Takeuchi \cite{Takeuchipre}\\\hline\hline
    $q$ &  $(2q/(2+q),q/2)$&$(2q/(2+q),q)$ & $(q/2,q)$\\ \hline 
    $2$ &  $(1,1)$ Napier&$(1,2)$ V. Riccati &$(1,2)$ V. Riccati\\
    $3$ &  $(6/5,3/2)$ \textbf{Theorem \ref{thm:6/53/2double}}&$(6/5,3)$ Miyakawa-Takeuchi \cite{Miyakawa-Takeuchi} &$(3/2,3)$ Dixon \cite{Dixon1890}\\
    $4$ &  $(4/3,2)$ Sato-Takeuchi \cite{Sato-Takeuchi2020} &$(4/3,4)$ Edmunds\ et\ al. \cite{Edmunds2012} &$(2,4)$ Fagnano\\ 
    $6$ &  $(3/2,3)$ Dixon \cite{Dixon1890} &$(3/2,6)$ Miyakawa-Takeuchi \cite{Miyakawa-Takeuchi} &$(3,6)$ Miyakawa-Takeuchi \cite{Miyakawa-Takeuchi}\\\hline
\end{tabular}
}
\end{center}
\end{table}

The multiple-angle formulas in the following theorem assure that GTFs for $(2q/(2+q),q/2)$ 
can be represented in terms of GTFs for $(2q/(2+q),q)$. 
Moreover, the counterparts to GHFs are obtained as their dual inequalities.

\begin{thm}\label{thm:sin_new_double2}
Let $0<q<\infty$. Then, for $x\in[0,\pi_{2q/(2+q),q/2}/(2^{2/q+1}))=[0,\pi_{2q/(2+q),q}/2)$,
\begin{align*}
\sin_{2q/(2+q),q/2}(2^{2/q}x)=&\frac{2^{2/q}\sin_{2q/(2+q),q}x}{(1+\sin_{2q/(2+q),q}^{q/2}x)^{2/q}},\\
\cos_{2q/(2+q),q/2}(2^{2/q}x)=&\left(\frac{1-\sin_{2q/(2+q),q}^{q/2}x}{1+\sin_{2q/(2+q),q}^{q/2}x}\right)^{1/q+1/2}.
\end{align*}
Moreover, for same $x$,
\begin{align*}
\sinh_{2q/(2+q),q/2}(2^{2/q}x)&=2^{2/q}\sinh_{2,q}x(\cosh_{2,q}x+\sinh_{2,q}^{q/2}x)^{2/q},\\
\cosh_{2q/(2+q),q/2}(2^{2/q}x)&=(\cosh_{2,q}x+\sinh_{2,q}^{q/2}x)^{2/q+1}.
\end{align*}
\end{thm}
\begin{proof}
The former half is shown as follows. Let $y\in[0,\infty)$. Setting $t^q=u^q/(4(1-u^{q/2}))$ in
$$\sinh_{2,q}^{-1}y=\int^y_0\frac{dt}{\sqrt{1+t^q}},$$
we have
\begin{align*}
\sinh_{2,q}^{-1}y=&2^{-2/q-1}\int^{2^{2/q}y/(y^{q/2}+\sqrt{y^q+1})^{2/q}}_0\frac{2(1-u^{q/2})^{1/2}}{2-u^{q/2}}\cdot\frac{2-u^{q/2}}{(1-u^{q/2})^{1/q+1}}\ du\\
=&2^{-2/q}\int^{2^{2/q}y/(y^{q/2}+\sqrt{y^q+1})^{2/q}}_0\frac{du}{(1-u^{q/2})^{1/q+1/2}};
\end{align*}
that is,
\begin{align}
\sinh_{2,q}^{-1}y=2^{-2/q}\sin_{2q/(2+q),q/2}^{-1}\left(\frac{2^{2/q}y}{(y^{q/2}+\sqrt{y^q+1})^{2/q}}\right).\label{eq:sinh2qsin2q2+q}
\end{align}
Letting $y\to\infty$ in \eqref{eq:sinh2qsin2q2+q} and using $r_q(2)=2q/(2+q)$, we get
$$\frac{\pi_{2q/(2+q),q}}{2}=\frac{\pi_{2q/(2+q),q/2}}{2^{2/q+1}}.$$
From \eqref{eq:sinh2qsin2q2+q}, we see that for $x\in[0,\pi_{2q/(2+q),q/2}/(2^{2/q+1}))=[0,\pi_{2q/(2+q),q}/2)$,
\begin{align*}
\sin_{2q/(2+q),q/2}(2^{2/q}x)=&\frac{2^{2/q}\sinh_{2,q}x}{(\sinh_{2,q}^{q/2}x+\cosh_{2,q}x)^{2/q}}=\frac{2^{2/q}\tamh_{2,q}x}{(\tamh_{2,q}^{q/2}x+1)^{2/q}}.
\end{align*}
Theorem \ref{lem:GTFGHFrelation} with $r_q(2)=2q/(2+q)$ shows that the right-hand side becomes
\begin{align*}
\frac{2^{2/q}\sin_{2q/(2+q),q}x}{(\sin_{2q/(2+q),q}^{q/2}x+1)^{2/q}}.
\end{align*}
The formula of $\cos_{2q/(q+2),q/2}$ immediately follows from \eqref{eq:Tpythagoras}.

The latter half is proved as follows. By Theorem \ref{lem:GTFGHFrelation} 
with $r_{q/2}{(2q/(2+q))}=2q/(2+q)$ and the former half,
\begin{align*}
\sinh_{2q/(2+q),q/2}(2^{2/q}x)=\frac{\sin_{2q/(2+q),q/2}(2^{2/q}x)}{\cos_{2q/(2+q),q/2}^{4/(2+q)}(2^{2/q}x)}=\frac{2^{2/q}\sin_{2q/(2+q),q}x}{(1-\sin_{2q/(2+q),q}^{q/2}x)^{2/q}}.
\end{align*}
Theorem \ref{lem:GHFGTFrelation} with $r_q(2q/(2+q))=2$ shows that the right-hand side becomes
$$2^{2/q}\sinh_{2,q}x(\cosh_{2,q}x+\sinh_{2,q}^{q/2}x)^{2/q}.$$
The formula of $\cosh_{2q/(q+2),q/2}$ immediately follows from \eqref{eq:Tpythagoras}.
\end{proof}

\begin{rem}
If $q=2$, then
the formulas of $\sin_{2q/(2+q),q/2}$ and $\sinh_{2q/(2+q),q/2}$ are 
\begin{gather*}
1-e^{-2x}=\frac{2\tanh{x}}{1+\tanh{x}},\\
e^{2x}-1=2\sinh{x}(\cosh{x}+\sinh{x}).
\end{gather*}
\end{rem}

The following double-angle formula is proved by
\cite{Miyakawa-Takeuchi}*{Theorem 3.8}.

\begin{lem}[\cite{Miyakawa-Takeuchi}]\label{lem:sin6/53double}
For $x\in[0,\pi_{6/5,3}/4)$,
$$\sin_{6/5,3}(2x)=\frac{4\cos_{6/5,3}^{1/5}x(3\cos_{6/5,3}^{3/5}x+1)(1-\cos_{6/5,3}^{3/5}x)^{1/3}}{\left(16\cos_{6/5,3}^{3/5}x+(3\cos_{6/5,3}^{3/5}x+1)^3(1-\cos_{6/5,3}^{3/5}x)\right)^{2/3}}.$$
\end{lem}

Now, we are in a position to show the double-angle formula of $\sin_{6/5,3/2}$.

\begin{thm}\label{thm:6/53/2double}
For $x\in[0,\pi_{6/5,3/2}/4)$,
$$\sin_{6/5,3/2}(2x)=(\Theta \circ \Phi \circ \Psi)(\cos_{6/5,3/2}{x}),$$
where
\begin{gather*}
\Theta(x)=\left(\frac{2x}{1+x}\right)^{2/3},\\
\Phi(x)=\frac{8\sqrt{x(3x+1)^3(1-x)}}{16x+(3x+1)^3(1-x)},\\
\Psi(x)=\frac{2x^{3/5}}{1+x^{6/5}}.
\end{gather*}
\end{thm}
\begin{proof}
From Theorem \ref{thm:sin_new_double2} with $q=3$, for $x\in[0,\pi_{6/5,3/2}/(2^{5/3}))=[0,\pi_{6/5,3}/2)$,
$$\sin_{6/5,3/2}(2^{2/3}x)=\frac{2^{2/3}\sin_{6/5,3}x}{(1+\sin_{6/5,3}^{3/2}x)^{2/3}}
=\Theta(\sin_{6/5,3}^{3/2}x);$$
hence,
$$\sin_{6/5,3}^{3/2}x=\Theta^{-1}(\sin_{6/5,3/2}(2^{2/3}x))
=\frac{\sin_{6/5,3/2}^{3/2}(2^{2/3}x)}{2-\sin_{6/5,3/2}^{3/2}(2^{2/3}x)}.$$
Thus, from \eqref{eq:Tpythagoras},
\begin{align}
\cos_{6/5,3}^{3/5}x=\frac{2\cos_{6/5,3/2}^{3/5}(2^{2/3}x)}{1+\cos_{6/5,3/2}^{6/5}(2^{2/3}x)}
=\Psi(\cos_{6/5,3/2}{(2^{2/3}x)}). \label{eq:c6/32to6/53/2}
\end{align}

Now,\ let $x\in[0,\pi_{6/5,3/2}/4)$ and $y:=x/(2^{2/3})$.
It follows from Theorem \ref {thm:sin_new_double2} with $q=3$ 
that since $2y\in[0,\pi_{6/5,3/2}/(2^{5/3}))=[0,\pi_{6/5,3}/2)$, we get
\begin{align}
\sin_{6/5,3/2}(2x)=\sin_{6/5,3/2}(2^{2/3}\cdot 2y)=\Theta(\sin_{6/5,3}^{3/2}{(2y)}).
\label{eq:s6/53/2(2x)Tos6/532y}
\end{align}
Here,\ Lemma \ref{lem:sin6/53double} and \eqref{eq:c6/32to6/53/2} yield
\begin{align*}
\sin_{6/5,3}^{3/2}(2y)=\Phi(\cos_{6/5,3}^{3/5}y)=\Phi(\Psi(\cos_{6/5,3/2}(2^{2/3}y)))
=\Phi(\Psi(\cos_{6/5,3/2}x)).
\end{align*}
Therefore, from \eqref{eq:s6/53/2(2x)Tos6/532y}, we have
\begin{align*}
\sin_{6/5,3/2}(2x)=\Theta(\Phi(\Psi(\cos_{6/5,3/2}x)))
=(\Theta \circ \Phi \circ \Psi)(\cos_{6/5,3/2}x).
\end{align*}
The proof is completed.
\end{proof}




\end{document}